\documentclass[a4paper,11pt,fleqn]{amsart}
\usepackage{natbib}
\usepackage{mathscinet}
\usepackage{mathrsfs}
\usepackage{hyperref}
\hypersetup{
    pdftitle={Gaussian queues in light and heavy traffic},
    pdfsubject={Probability Theory},
    pdfauthor={Kamil Kosi\'nski},
    pdfdisplaydoctitle=true,
}

\DeclareMathOperator{\Cov}{\mathbb{C}ov}
\DeclareMathOperator{\id}{id}
\DeclareMathOperator{\diam}{diam}

\newcommand{\ee}{\mathbb E}
\newcommand{\nn}{\mathbb N}
\newcommand{\pp}{\mathbb P}
\newcommand{\rr}{\mathbb R}
\newcommand{\ttt}{\mathbb T}
\newcommand{\vb}{\vspace{3mm}}

\newcommand{\half}{\frac{1}{2}}
\newcommand{\toi}{\to\infty}

\newcommand{\de}{\stackrel{d}{=}}
\newcommand{\dto}{\stackrel{d}{\to}}

\newcommand{\prob}[1]{\pp\left(#1\right)}
\newcommand{\rv}{\mathscr{RV}}

\newtheorem{theorem}{Theorem}
\newtheorem{lem}{Lemma}
\newtheorem{cor}{Corollary}
\newtheorem{prop}{Proposition}
\theoremstyle{remark}
\newtheorem{rem}{Remark}

\renewenvironment{proof}[1][\proofname]{\par \normalfont \trivlist
 \item[\hskip\labelsep\itshape #1]\ignorespaces
}{%
 \hspace*{\fill}$\Box$ \endtrivlist
}
\renewcommand{\proofname}{\noindent {\bf Proof}}

\begin{document}
\date{\today}
\title[Gaussian queues in light and heavy traffic]{Gaussian queues \\in  light and heavy traffic}

\author{K.\ D\polhk{e}bicki}
\address{Instytut Matematyczny, University of
Wroc\l aw, pl.\ Grunwaldzki 2/4, 50-384 Wroc\l aw, Poland.}
\email{Krzysztof.Debicki@math.uni.wroc.pl}
\thanks{KD was supported by
MNiSW Grant N N201 394137 (2009-2011) and by a travel grant from
NWO (Mathematics Cluster STAR)}

\author{K.M.\ Kosi\'nski}
\address{Korteweg-de Vries Institute for Mathematics,
University of Amsterdam, the Netherlands; E{\sc urandom},
Eindhoven University of Technology}
\email{K.M.Kosinski@uva.nl}
\thanks{KK was supported by NWO grant 613.000.701.}

\author{M.\ Mandjes}
\address{Korteweg-de Vries Institute for Mathematics,
University of Amsterdam, the Netherlands; E{\sc urandom},
Eindhoven University of Technology, the Netherlands; CWI,
Amsterdam, the Netherlands}
\email{M.R.H.Mandjes@uva.nl}

\bibliographystyle{plainnat}
\setcitestyle{numbers}

\begin{abstract}
In this paper we investigate Gaussian queues in the light-traffic
and in the heavy-traffic regime.
Let $Q^{(c)}_X\equiv\{Q^{(c)}_X(t):t\ge 0\}$ denote a stationary
buffer content process for a fluid queue fed by the centered Gaussian
process $X\equiv\{X(t):t\in\rr\}$ with stationary increments,
$X(0)=0$, continuous sample paths and variance function
$\sigma^2(\cdot)$. The system is drained with
a constant rate $c>0$, so that for any $t\ge 0$,
\[
Q_{X}^{(c)}(t)=\sup_{-\infty<s\le t}\left(X(t)-X(s) -c(t-s)\right).
\]
We study $Q^{(c)}_X\equiv\{Q_X^{(c)}(t):t\ge0\}$ in the
regimes $c\to 0$ (heavy traffic) and $c\to\infty$ (light
traffic). We show for both limiting
regimes that, under mild regularity conditions on $\sigma$, there exists a normalizing function $\delta(c)$ such that
$Q^{(c)}_X(\delta(c)\cdot)/\sigma(\delta(c))$ converges
to $Q^{(1)}_{B_H}(\cdot)$ in $C[0,\infty)$, where $B_H$ is a fractional Brownian motion with suitably chosen Hurst parameter $H$.
\end{abstract}

\maketitle

\section{Introduction}
\label{s.intro} A substantial research effort has been devoted to
the analysis of queues with Gaussian input, often also called {\it
Gaussian queues} \cite{Mandjes07,Mannersalo02,Norros94}. The
interest in this model can be explained from the fact that the
Gaussian input model is highly flexible in terms of incorporating a
broad set of correlation structures and, at the same time, adequately approximates various real-life systems. A key result
in this area is \cite{Taqqu97}, where it is shown that large
aggregates of Internet sources converge to a fractional Brownian
motion (being a specific Gaussian process).

The setting considered in this paper is that of a centered Gaussian
process $X\equiv\{X(t):t\in\mathbb R\}$ with stationary increments, $X(0)=0$, continuous sample paths
and variance function $\sigma^2(\cdot)$, equipped with a
deterministic, linear drift with rate $c>0$, reflected at 0:
\[Q_X^{(c)}(t)=\sup_{-\infty<s\le
t}(X(t)-X(s)-c(t-s)).\] The resulting {\it stationary workload
process} can be regarded as a {\it queue} \cite{Reich58}. The
objective of the paper is to study
$Q^{(c)}_X\equiv\{Q_X^{(c)}(t):t\ge0\}$  in the limiting regimes
$c\to0$ (heavy traffic) and $c\to\infty$ (light traffic).

Under mild conditions on the variance function $\sigma^2(\cdot)$,
$Q_X^{(c)}$ is a properly defined, almost surely (a.s.) finite
stochastic process. However, if $c\to 0$, then $Q_X^{(c)}(t)$
grows to infinity (in a distributional sense), for any $t\ge0$. The branch of
queueing theory investigating {\it how fast} $Q_{X}^{(c)}$ grows to
infinity (as $c\to 0$) is commonly referred to as the domain
of {\it heavy-traffic approximations}. In many situations this
regime allows manageable expressions for performance metrics that
are, under `normal' load conditions, highly complex or even
intractable, see for instance the seminal paper by \citet{Kingman61}
on the classical single-server queue. Since then, a similar approach
has been followed in various other settings, see, e.g.,
\cite{Boxma99, Prokhorov63, Resnick00, Szczotka04, Whitt71} and many
other papers.

Analogously, one can ask what happens in the {\it light-traffic}
regime, i.e.,  $c\to \infty$; then evidently $Q_X^{(c)}$ decreases
to zero. So far, hardly any attention has been paid to the
light-traffic and heavy-traffic regimes for Gaussian queues. An
exception is \citet{Debicki04a}, where the focus is on a special
family of Gaussian processes, in a specific heavy-traffic setting.
The primary contribution of the present paper concerns the analysis
of $Q_X^{(c)}$ under both limiting regimes, for quite a broad class
of Gaussian input processes $X$.

\vb

We now give a somewhat more detailed introduction to the material
presented in this paper. 
It is  well known that under the assumption that
$\sigma(\cdot)$ varies regularly at infinity with parameter
$\alpha\in(0,1)$, for any function $\delta$ such that
$\delta(c)\toi$ as $c\to0$, there is convergence to fractional
Brownian motion in the heavy-traffic regime:
\begin{equation}
\label{eq:introHT}
\frac{X(\delta(c)\cdot)}{\sigma(\delta(c))}\dto B_\alpha(\cdot),$ as $c\to0.
\end{equation}
We shall show that an analogous statement holds in the light-traffic
regime, that is, if $\sigma(\cdot)$ varies regularly at zero with
parameter $\lambda\in(0,1)$ (i.e., $x\mapsto\sigma(1/x)$ varies regularly at
infinity with parameter $-\lambda$), then for any function $\delta$ such
that $\delta(c)\to0$ as $c\toi$,
\begin{equation}
\label{eq:introLT}
\frac{X(\delta(c)\cdot)}{\sigma(\delta(c))}\dto B_\lambda(\cdot),$ as $c\toi.
\end{equation}
Assuming that $X$ satisfies some minor additional conditions, both \eqref{eq:introHT} and \eqref{eq:introLT} apply in $C(\rr)$, the space of all continuous functions
on $\rr$.

Our paper shows that the statements \eqref{eq:introHT} and
\eqref{eq:introLT}, which relate to the input processes, carry over
to the corresponding stationary buffer content processes
$Q^{(c)}_X$. That is, we identify, under specific conditions, a function
$\delta(\cdot)$ such that
\[
\frac{Q_X^{(c)}\left(\delta(c)\cdot\right)}{\sigma(\delta(c))}\dto
Q_{B_\alpha}^{(1)}(\cdot),$ as $c\to0
\]
and
\[
\frac{Q_X^{(c)}\left(\delta(c)\cdot\right)}{\sigma(\delta(c))}\dto
Q_{B_\lambda}^{(1)}(\cdot),$ as $c\toi,
\]
both in the space $C[0,\infty)$ of all continuous functions
on $[0,\infty)$.

This paper is organized as follows. In \autoref{s.prel} we introduce
the notation and give some preliminaries. \autoref{sec:HT} presents
the results for the heavy-traffic regime, whereas \autoref{sec:LT}
covers the light-traffic regime. We give the proofs of the main
theorems, (i.e., \autoref{th.heavy} and \autoref{thm:LT}) in
\autoref{sec:proof}.

\section{Preliminaries}
\label{s.prel} In this paper we use the following notation. By
$\id:\rr\to\rr$ we shall denote the identity operator on $\rr$, that
is, $\id(t)=t$ for every $t\in\rr$. We write $f(x)\sim g(x)$ as $x\to
x_0\in[0,\infty]$ when $\lim_{x\to x_0}f(x)/g(x)=1$. Let
$\rv_\infty(\alpha)$ and $\rv_0(\lambda)$ denote the class of
regularly varying functions at infinity with parameter $\alpha$ and
at zero with parameter $\lambda$, respectively. That is, for a non-negative measurable functions $f,g$ on $[0,\infty)$, $f\in\rv_\infty(\alpha)$ if for all $t>0$,
$f(t x)/f(x)\to t^{\alpha}$ as $x\toi$;
$g\in\rv_0(\lambda)$ if for all $t>0$,
$g(t x)/g(x)\to t^\lambda$ as $x\to0$.

\subsection{Spaces of continuous functions}
We refer to \citet{Billingsley99} for the details of this subsection.
For any $T>0$, let 
$C[-T,T]$ be the space of all continuous
functions $f:[-T,T]\to\rr$. Equip
$C[-T,T]$ with the topology of uniform convergence, i.e., the
topology generated by the norm $\|f\|_{[-T,T]}:=\sup_{t\in
[-T,T]}|f(t)|$ under which $C[-T,T]$ is a separable Banach space. 
Therefore, by
Prokhorov's theorem, weak convergence
of random elements $\{X^{(c)}\}$ of $C[-T,T]$ as $c\toi$ is implied by convergence of finite-dimensional distributions and tightness.
A family $\{X^{(c)}\}$ in $C[-T,T]$ is tight if and only if for each positive $\varepsilon$, there exists
an $a$ and $c_0$ such that
\begin{equation}
\label{eq:tight1}
\prob{|X^{(c)}(0)|\ge a}\le\varepsilon,$ for all $c\ge c_0;
\end{equation}
and, for any $\eta>0$,
\begin{equation}
\label{eq:modulus}
\lim_{\zeta\to0}\limsup_{c\to\infty}\prob{\sup_{\substack{|t-s|\le\zeta\\s,t\in[-T,T]}}
\left|X^{(c)}(t)-X^{(c)}(s)\right|\ge\eta}=0.
\end{equation}
For notational convenience, we leave out the requirement
$s,t\in[-T,T]$ explicitly in the remainder of this paper.

Finally, let $C(\rr)$ be the space of all functions
$f:\rr\to\rr$ such that $f_{|[-T,T]}\in C[-T,T]$ for all $T>0$. 
The above definitions
extend in an obvious way to $C[0,T]$, $C[0,\infty)$ and convergence as $c\to0$.

For $\gamma\ge0$, let $\Omega^\gamma$ be the space of all continuous
functions $f:\rr\to\rr$ such that $\lim_{t\to\pm\infty}
f(t)/(1+|t|^\gamma)=0$. Equip $\Omega^\gamma$ with the topology
generated by the norm
$\|f\|_{\Omega^\gamma}:=\sup_{t\in\rr}|f(t)|/(1+|t|^\gamma)$ under
which $\Omega^\gamma$ is a separable Banach space, so that
Prokhorov's theorem applies. 
The following property can be found in \cite[Lemma 3]{Buldygin99} or \cite[Lemma 4]{Dieker05}.
\begin{prop}
\label{prop:convO}
Let a family of random elements $\{X^{(c)}\}$ on $\Omega^\gamma$ be given. Suppose that
the image of $\{X^{(c)}\}$ under the projection mapping $p_T:\Omega^\gamma\to C[-T,T]$ is tight
in $C[-T,T]$ for all $T>0$. Then $\{X^{(c)}\}$ is tight in $\Omega^\gamma$ if and only if for any
$\eta>0$,
\begin{equation}
\label{eq:prop:convO}
\lim_{T\toi}\limsup_{c\toi}\prob{\sup_{|t|\ge T}\frac{|X^{(c)}(t)|}{1+|t|^\gamma}\ge\eta}=0.
\end{equation}
\end{prop}

\subsection{Fluid Queues}
Let $Q^{(c)}_X\equiv\{Q^{(c)}_X(t):t\ge 0\}$ denote a stationary
buffer content process for a fluid queue fed by a centered Gaussian
process $X\equiv\{X(t):t\in\rr\}$ with stationary increments,
$X(0)=0$, continuous sample paths and variance function
$\sigma^2(\cdot)$. The system is drained with
a constant rate $c>0$, so that for any $t\ge 0$,
\[
Q_{X}^{(c)}(t)=\sup_{-\infty<s\le t}\left(X(t)-X(s) -c(t-s)\right).
\]

Additionally, an equivalent representation for $Q_{X}^{(c)}(t)$
holds \cite[p. 375]{Rob03}:
\begin{equation}
\label{eq:defQ}
Q^{(c)}_X(t)=Q^{(c)}_X(0)+X(t)-ct+\max\left(0,\sup_{0<s<t}\left(-Q_X^{(c)}(0)-(X(s)-cs)\right)\right).
\end{equation}

Throughout the paper we say that $X$ satisfies:
\begin{itemize}
\item[{\bf C:}] if $\sigma^2(t)|\log|t||^{1+\varepsilon}$
has a finite limit as $t\to0$, for some $\varepsilon>0$;
\item[{\bf RV}$_0${\bf :}] if $\sigma\in\rv_0(\lambda)$, for $\lambda\in(0,1)$;
\item[{\bf RV}$_\infty${\bf :}] if $\sigma\in\rv_\infty(\alpha)$, for $\alpha\in(0,1)$;
\item[{\bf HT:}] if both {\bf C} and {\bf RV}$_\infty$ are satisfied.
\item[{\bf LT:}] if both {\bf RV}$_0$ and {\bf RV}$_\infty$ are satisfied.
\end{itemize}

\begin{rem} In our setting ($X$ has stationary increments), the assumption that $X$ is continuous is equivalent to the convergence of {\it Dudley integral}, see \autoref{subsect:ME}. This is immediately implied by condition {\bf C}, see \cite[Thm.\ 1.4]{Adler90}. 
However, the real importance of condition {\bf C} lies in the fact that if in addition $X$ satisfies {\bf RV}$_\infty$, then $X$ also belongs to $\Omega^\gamma$, for every $\gamma>\alpha$. This is pointed out in 
\autoref{sec:HT}.
Finally, note that {\bf C} is met under {\bf RV}$_0$. Indeed, since $\sigma\in\rv_0(\lambda)$, then
$t\mapsto\sigma(1/t)$ belongs to $\rv_\infty(-\lambda)$, thus $\sigma^2(1/t)t^{\lambda}\to0$ as $t\toi$.
Equivalently, $\sigma^2(t)t^{-\lambda}\to 0$ as $t\to0$, implying $\lim_{t\to
0}\sigma^2(t)|\log|t||^{1+\varepsilon}=0$, for any fixed $\varepsilon>0$.
Furthermore, 
{\bf RV}$_\infty$ implies that $X(t)/t\to 0$ a.s., for $t\to\pm\infty$,
so that $Q_{X}^{(c)}$ is a properly defined stochastic process for any $c>0$, see \cite[Lemma 3]{Dieker05}. Lastly, the assumption that $X$ has continuous sample paths implies that $\sigma$ is continuous.
\end{rem}
Due to the stationarity of increments, all finite-dimensional distributions of $X$ are  specified by the variance function, since we have
\begin{equation}
\label{eq:cov}
\Cov(X(t),X(s))=\half\left(\sigma^2(s)+\sigma^2(t)-\sigma^2(|t-s|)\right).
\end{equation}
Recall that by $B_H\equiv\{B_H(t):t\in\rr\}$ we  denote fractional
Brownian motion with Hurst parameter $H\in(0,1)$, that is, a
centered Gaussian process with stationary increments, continuous
sample paths, $B_H(0)=0$ and covariance function
\begin{equation}
\label{eq:covB}
\Cov(B_H(t),B_H(s))=\half\left(|s|^{2H}+|t|^{2H}-|t-s|^{2H}\right).
\end{equation}
As mentioned in the introduction, if $c\to0$, then, for any $t$, 
$Q_{X}^{(c)}(t)\to\infty$ a.s., which is called the {\it heavy-traffic
regime}. On the other hand, if $c\toi$, then $Q_{X}^{(c)}(t)\to 0$
a.s., which is called the {\it light-traffic regime}.

\subsection{Metric entropy}
\label{subsect:ME} For any $\ttt\subset\mathbb R$ define the {\it
semimetric}
\[
d(t,s):=\sqrt{\ee|X(t)-X(s)|^2}=\sigma(|t-s|),$\:\:\: $t,s\in\ttt.
\]
We say that $S\subset\ttt$ is a $\vartheta$-net in $\ttt$ with
respect to the semimetric $d$, if for any $t\in\ttt$ there exists an
$s\in S$ such that $d(t,s)\le\vartheta$. The metric entropy $\mathbb
H_d(\ttt,\vartheta)$ is defined as $\log\mathbb
N_d(\ttt,\vartheta)$, where $\mathbb N_d(\ttt,\vartheta)$ denotes
the minimal number of points in a $\vartheta$-net in $\ttt$ with
respect to $d$. Later on we use the following proposition, see
\cite[Thm.\ 1.3.3]{Adler07} and \cite[Corollary 1.3.4]{Adler07},
respectively.
\begin{prop}\label{prop:ME}
There exists a universal constant $K$ such that for a $d$-compact set $\ttt$
\[
\ee \left(\sup_{t\in\ttt} X(t)\right)\le K\int_0^{\diam(\ttt)/2}
\sqrt{\mathbb H_d\left(\ttt,\vartheta\right)}\,\rm d\vartheta.
\]
and for all $\zeta>0$
\[
\ee\left( \sup_{\substack{(s,t)\in\ttt\times\ttt\\d(s,t)<\zeta}}
|X(t)-X(s)|\right)\le K\int_0^{\zeta} \sqrt{\mathbb
H_d\left(\ttt,\vartheta\right)}\,\rm d\vartheta.
\]
\end{prop}
The quantity $\int_0^\infty\sqrt{\mathbb
H_d\left(\ttt,\vartheta\right)}\,\rm d\vartheta$ is called the Dudley integral.

\section{Main Results}
In this section we formulate the result for the heavy-traffic and
light-traffic regime, respectively. It is emphasized that these
results are highly symmetric. 
Let us first introduce a function $\delta$, such that for every $c>0$
\begin{equation}
\label{eq:defdelta}
\frac{c\delta(c)}{\sigma(\delta(c))}=1.
\end{equation}
By the continuity of $\sigma$, we can choose $\delta$ as $\delta(c)=\inf\{x>0:x/\sigma(x)=1/c\}$. From the definition of $\delta$ it follows that $\delta\in\rv_0(1/(\alpha-1))$ under {\bf RV}$_\infty$ and $\delta\in\rv_\infty(1/(\lambda-1))$ under {\bf RV}$_0$.

\subsection{Heavy-traffic Regime}
\label{sec:HT} In the heavy-traffic regime we are interested in the
analysis of $Q_X^{(c)}$ as $c\to0$, under the assumption that $X$
satisfies {\bf HT}. The following statement follows from
\cite[Thms.\ 5 and 6]{Dieker05}.
\begin{prop}
\label{thm:convT} If $X$ satisfies {\bf HT}, then
\[
\frac{X(\delta(c)\cdot)}{\sigma(\delta(c))}\dto B_\alpha(\cdot),$ as $c\to0,
\]
in $C(\rr)$ and $\Omega^\gamma$, for any $\gamma>\alpha$.
\end{prop}
In fact, \autoref{thm:convT} holds for any function $\delta(c)$ such that $\delta(c)\toi$ as $c\to 0$. Condition {\bf C} (which is one of the requirements of {\bf HT}) plays a crucial role in proving tightness both in $C[-T,T]$, for some $T>0$, and in $\Omega^\gamma$. 

Combining \autoref{thm:convT} with the definition of $\delta$ leads to the following statement.
\begin{cor}
\label{cor:conv}
If $X$ satisfies {\bf HT}, then
\[
\frac{X(\delta(c)\cdot) -c\delta(c)\id(\cdot)}{\sigma(\delta(c))}
\dto
B_{\alpha}(\cdot)-\id(\cdot)$ as $c\to0,
\]
in $C(\rr)$.
\end{cor}
Now we are in the position to present the main result of this subsection.
\label{s.main}
\begin{theorem}
\label{th.heavy}
If $X$ satisfies {\bf HT}, then
\begin{equation}
\label{eq:thmHT}
\frac{Q^{(c)}_X(\delta(c)\cdot)}{\sigma(\delta(c))} \dto Q_{B_{\alpha}}^{(1)}(\cdot)$ as $c\to0,
\end{equation}
in $C[0,\infty)$.
\end{theorem}
We postpone the proof of \autoref{th.heavy} to \autoref{sec:proof}.
\begin{rem}
Theorem \ref{th.heavy} extends the findings of \cite[Theorem
3.2]{Debicki04a} where, under the heavy-traffic regime, the weak
convergence in $C[0,\infty)$ of
$Q^{(c)}_X(\delta(c)\cdot)/\sigma(\delta(c))$ as $c\to0$ was
obtained for the class of input processes having differentiable
sample paths a.s., i.e., of the form $X(t)=\int_0^t Z(s){\rm d}s$,
where $\{Z(s):s\ge0\}$ is a stationary centered Gaussian process
whose variance function satisfies specific regularity conditions.
\end{rem}

\subsection{Light-traffic Regime}
\label{sec:LT} In the light-traffic regime we analyze the
convergence of $Q_X^{(c)}$ as $c\to\infty$, under the assumption
that $X$ satisfies {\bf LT}. We begin by stating the counterpart of
\autoref{thm:convT}.
\begin{prop}
\label{thm:convLT} If $X$ satisfies {\bf RV}$_0$, then 
\[
\frac{X(\delta(c)\cdot)}{\sigma(\delta(c))}\dto B_\lambda(\cdot),$ as $c\toi,
\]
in $C(\rr)$. If, moreover, $X$ satisfies {\bf LT}, then the
convergence also holds in $\Omega^\gamma$, for any
$\gamma>\max\{\lambda,\alpha\}$.
\end{prop}
Analogously to \autoref{thm:convT}, \autoref{thm:convLT} holds
for any function $\delta(c)$ such that $\delta(c)\to0$ as $c\toi$.
As in the heavy-traffic case, combining \autoref{thm:convLT} with the definition of $\delta$ leads to the counterpart of \autoref{cor:conv}.
\begin{cor}
\label{cor:convLT}
If $X$ satisfies {\bf RV}$_0$, then
\[
\frac{X(\delta(c)\cdot) -c\delta(c)\id(\cdot)}{\sigma(\delta(c))}
\dto
B_{\lambda}(\cdot)-\id(\cdot)$ as $c\toi,
\]
in $C(\rr)$.
\end{cor}
The main result of this subsection is now stated as follows.
\label{s.main.LT}
\begin{theorem}
\label{thm:LT}
If $X$ satisfies {\bf LT}, then
\begin{equation}
\label{eq:thmLT}
\frac{Q^{(c)}_X(\delta(c)\cdot)}{\sigma(\delta(c))} \dto Q_{B_{\lambda}}^{(1)}(\cdot)$ as $c\toi,
\end{equation}
in $C[0,\infty)$.
\end{theorem}
We postpone the proof of \autoref{thm:convLT} and \autoref{thm:LT} to \autoref{sec:proof}.
\begin{rem}
The assumption {\bf LT} excludes the class of input processes of the
structure $X(t)=\int_0^t Z(s){\rm d}s$, with $\{Z(s):s\ge0\}$ being
a centered stationary Gaussian process with continuous sample paths
a.s.\ (since $\lambda=1$ in this case). In \cite[Theorem
4.1]{Debicki04a} it was shown that, for this class of Gaussian
processes, $Q^{(c)}_X(0)/\sigma(\delta(c))$ does {\it not}
converge weakly to $Q_{B_{\lambda}}^{(1)}(0)$ as $c\toi$.
\end{rem}

\section{Proofs}
\label{sec:proof} In this section we prove our results, but we start
by presenting an auxiliary result.
\begin{lem}
\label{lem:sigma} If $X$ satisfies {\bf LT}, then for any
$\epsilon>0$, there exist constants $C,a>0$, such that for all $x\le
a$ and $t>0$,
\[
\frac{\sigma(tx)}{\sigma(x)}\le C\times\left\{
\begin{array}{cc}
t^\ell & t\le 1,\\
t^u & t> 1,
\end{array}
\right.
\]
where $\ell:=\min\{\lambda-\epsilon,\alpha+\epsilon\}$ and
$u:=\max\{\alpha+\epsilon,\lambda+\epsilon\}$.
\end{lem}

\begin{proof}
Take any $\epsilon>0$, then because $\sigma\in\rv_0(\lambda)$,
there exists an $a\le 1$ such that
\begin{equation}
\label{eq:bound1} \frac{\sigma(tx)}{\sigma(x)}\le 2
t^{\lambda-\epsilon}$, for all $x\le a$ and $tx\le a.
\end{equation}
Moreover, there exists a constant $K_1$ such that $\sigma(x)\ge K_1
x^{\lambda+\epsilon}$ for all $x\le a$.

Because $\sigma\in\rv_\infty(\alpha)$,  there exist constants
$A,K_2>0$ such that $\sigma(x)\le K_2x^{\alpha+\epsilon}$ for all
$x\ge A$. Because $\sigma$ is continuous, we can in fact find a
$K_2$ such that $\sigma(x)\le K_2x^{\alpha+\epsilon}$ for all $x\ge
a$. Therefore
\[
\frac{\sigma(tx)}{\sigma(x)}\le \frac{K_2(tx)^{\alpha+\epsilon}}{K_1
x^{\lambda+\epsilon}}=:Kt^{\alpha+\epsilon}x^{\alpha-\lambda}$, for
all $x\le a$ and $tx\ge a.
\]
Note that, if $\alpha-\lambda\ge0$, then we have
\begin{equation}
\label{eq:bound2} \frac{\sigma(tx)}{\sigma(x)}\le
Ka^{\alpha+\epsilon}t^{\alpha+\epsilon}$, for all $x\le a$ and $tx\ge a.
\end{equation}
If $\alpha-\lambda<0$, then
\begin{equation}
\label{eq:bound3} \frac{\sigma(tx)}{\sigma(x)}\le
Ka^{\alpha-\lambda}t^{\lambda+\epsilon}$, for all $x\le a$ and
$tx\ge a.
\end{equation}
Combining \eqref{eq:bound1}--\eqref{eq:bound3}, we conclude that
there exists a constant $C>0$, such that
\[
\frac{\sigma(tx)}{\sigma(x)} \le
C\max\left\{t^{\lambda-\epsilon},t^{\alpha+\epsilon},t^{\lambda+\epsilon}\right\}$,
for all $x\le a$ and all $t>0.
\]
\end{proof}

\vb

In what follows, we will use the following notation. Let
\[
X^{(c)}(t):=\frac{X(\delta(c)t)}{\sigma(\delta(c))}
\]
and denote the variance of $X^{(c)}$ by $(\sigma^{(c)})^2$, that is,
\[
\sigma^{(c)}(t):=\frac{\sigma(\delta(c)t)}{\sigma(\delta(c))}.
\]
\begin{proof}[{\bf Proof of \autoref{thm:convLT}}]
We begin by showing the convergence in $C(\rr)$. To this end, we
need to show the convergence in $C[-T,T]$ for any fixed $T>0$. 

\vb

{\it Convergence in $C[-T,T]$}: 
From the fact that
$\sigma\in\rv_0(\lambda)$, it is immediate that the
finite-dimensional distributions of $X^{(c)}$ converge in
distribution to $B_\lambda$ as $c\toi$, cf.\
\eqref{eq:cov}-\eqref{eq:covB}, which also implies \eqref{eq:tight1}. Therefore, 
the weak convergence of
$X^{(c)}$ in $C[-T,T]$ follows after showing \eqref{eq:modulus}.

By the Uniform Convergence Theorem, see \cite[Thm.\ 1.5.2]{Bingham87}, for any
$t\in(0,\zeta]$, we have $\sigma^{(c)}(t)\le 2\zeta^{\lambda}$.
Thus, \autoref{prop:ME} yields, for some universal constant
$K>0$,
\begin{align*}
\prob{\sup_{|s-t|\le\zeta}\left|X^{(c)}(t)-X^{(c)}(s)\right|\ge\eta}
&\le
\prob{\sup_{\sigma^{(c)}(|s-t|)\le 2\zeta^{\lambda}}\left|X^{(c)}(t)-X^{(c)}(s)\right|\ge\eta}\\
&\le
\frac{1}{\eta}\:\ee\left(\sup_{\sigma^{(c)}(|s-t|)\le 2\zeta^{\lambda}}\left|X^{(c)}(t)-X^{(c)}(s)\right|\right)\\
&\le \frac{K}{\eta}\int_0^{2\zeta^{\lambda}}\sqrt{\mathbb
H^{(c)}([-T,T],\vartheta)}\,\rm d\vartheta,
\end{align*}
where $\mathbb H^{(c)}([-T,T],\cdot)$ is the metric entropy induced
by $\sigma^{(c)}$.

By Potter's bound \cite[Thm.\ 1.5.6]{Bingham87} for any
$\epsilon,\zeta>0$, $\epsilon<\lambda$ and $t\in(0,\zeta]$ and
sufficiently large $c$ (corresponding to small $\delta(c)$), we have
$\sigma^{(c)}(t)\le 2 t^{\lambda-\epsilon}$. Hence
\[
\mathbb H^{(c)}([-T,T],\vartheta)\le \mathbb H_{\tilde
d}\left([-T,T],\frac{\vartheta}{2}\right),
\]
where $\tilde d$ is a semimetric such that $\tilde
d(s,t)=|t-s|^{\lambda-\epsilon}$. The inverse of $x\mapsto
x^{\lambda-\epsilon}$ is given by $x\mapsto
x^{1/(\lambda-\epsilon)}$, so that
\[
\mathbb H_{\tilde d}([-T,T],\vartheta)\le
\log\left(\frac{T}{\vartheta^{1/(\lambda-\epsilon)}}+1\right) \le
C\log\left(\frac{1}{\vartheta}\right),
\]
for some constant $C>0$ and $\vartheta>0$ small. It follows that
\begin{align*}
\int_0^{2\zeta^{\lambda}}\sqrt{\mathbb
H^{(c)}([-T,T],\vartheta)}\,{\rm d}\vartheta 
&\le 
\sqrt
C\int_0^{2\zeta^{\lambda}}\sqrt{\log\left(\frac{2}{\vartheta}\right)}\,{\rm
d}\vartheta = 2\sqrt C\int_{\zeta^{-\lambda}}^\infty\frac{\sqrt{\log
\vartheta}}{\vartheta^2}\,{\rm d}\vartheta.
\end{align*}
Summarizing, we have
\[
\limsup_{c\toi}\prob{\sup_{|s-t|\le\zeta}\left|X^{(c)}(t)-X^{(c)}(s)\right|\ge\eta}
\le \frac{2K\sqrt
C}{\eta}\int_{{\zeta^{-\lambda}}}^\infty\frac{\sqrt{\log
\vartheta}}{\vartheta^2}\,{\rm d}\vartheta;
\]
we obtain \eqref{eq:modulus} by letting $\zeta\to0$.

\vb

{\it Convergence in $\Omega^\gamma$}: To show the convergence in
$\Omega^\gamma$, we need to verify \eqref{eq:prop:convO}. Observe
that
\begin{align*}
\prob{\sup_{t\ge e^k}
    \frac{|X^{(c)}(t)|}{1+t^\gamma}
    \ge \eta}
&\le
\frac{1}{\eta}\sum_{j=k}^\infty \frac{\ee\sup_{t\in[e^j,e^{j+1}]}|X^{(c)}(t)|}{1+e^{j\gamma}}\\
&\le \frac{1}{\eta}\sum_{j=k}^\infty \frac{\ee
|X^{(c)}(e^j)|}{1+e^{j\gamma}} +
\frac{2}{\eta}\sum_{j=k}^\infty \frac{\ee\sup_{t\in[e^j,e^{j+1}]}X^{(c)}(t)}{1+e^{j\gamma}}\\
&=:I_1(k)+I_2(k).
\end{align*}

$I_1(k)$ and $I_2(k)$ are dealt with separately. According to
\autoref{lem:sigma}, for large $c$ (that is, small $\delta(c)$), we
have
\[
\sigma^{(c)}(t)\le C\times\left\{
\begin{array}{cc}
t^\ell & t\le 1,\\
t^u & t> 1,
\end{array}
\right.
\]
where $\ell$ and $u$ can be chosen such that $\ell,u<\gamma$.
Therefore,
\[
I_1(k)\le \frac{1}{\eta}\sum_{j=k}^\infty
\frac{\sigma^{(c)}(e^j)}{1+e^{j\gamma}} \le
\frac{C}{\eta}\sum_{j=k}^\infty \frac{e^{ju}}{1+e^{j\gamma}},
\]
and the resulting upper bound tends to zero as $k\toi$.

Now focus on $I_2(k).$ For some universal constant $K>0$ and because
of the stationarity of the increments of $X$, \autoref{prop:ME}
yields that $I_2(k)$ is majorized by
\[
\frac{2K}{\eta} \sum_{j=k}^\infty \frac{\displaystyle
\int_0^\infty\sqrt{\mathbb H^{(c)}([e^j,e^{j+1}],\vartheta)}\,\rm
d\vartheta}{1+e^{j\gamma}} = \frac{2K}{\eta} \sum_{j=k}^\infty
\frac{\displaystyle \int_0^\infty\sqrt{\mathbb
H^{(c)}([0,e^j(e-1)],\vartheta)}\,\rm d\vartheta}{1+e^{j\gamma}}.
\]
We will estimate the integrals under the sum by splitting the integration area into $\vartheta\le 1$ and $\vartheta\ge 1$.
 
Observe that, for some constants $C_1,C_2>0$ (that is, not
depending on $j$),
\begin{align*}
\int_0^1\sqrt{\mathbb H^{(c)}([0,e^j(e-1)],\vartheta)}\,\rm
d\vartheta &\le
\int_0^1\sqrt{\log\left(\frac{e^j(e-1)}{2\vartheta^{{1}/{\ell}}}+1\right)}\,\rm d\vartheta\\
&\le
\int_0^1\sqrt{C_1+j+\frac{1}{\ell}\log\left(\frac{1}{\vartheta}\right)}\,\rm d\vartheta\\
&=
\ell e^{\ell(C_1+j)}\int_{C_1+j}^\infty\sqrt\vartheta e^{-\ell\vartheta}\,\rm d\vartheta\\
&\le \ell e^{\ell(C_1+j)}\int_{0}^\infty\sqrt\vartheta
e^{-\ell\vartheta}\,{\rm d}\vartheta= C_2e^{\ell j}.
\end{align*}
Recall that $\ell<\gamma$, so that
\[
\lim_{k\toi}\sum_{j=k}^\infty \frac{\displaystyle
\int_0^1\sqrt{\mathbb H^{(c)}([0,e^j(e-1)],\vartheta)}\,\rm
d\vartheta}{1+e^{j\gamma}}\le \frac{2K}{\eta} \lim_{k\toi}\sum_{j=k}^\infty
\frac{C_2e^{\ell j}}{1+e^{j\gamma}}=0.
\]
So it remains to show the analogous statement for the integration
interval $[1,\infty)$. Using a similar argumentation as the one
above, one can show that
\[
\int_1^\infty\sqrt{\mathbb H^{(c)}([0,e^j(e-1)],\vartheta)}\,{\rm
d}\vartheta \le C_3 e^{uj},
\]
for some constant $C_3>0$, from which the claim is readily obtained.
\end{proof}

\vb

Since the proof of \autoref{th.heavy} is analogous to the proof of
\autoref{thm:LT}, we choose to focus on the light-traffic case only.

\vb

\begin{proof}[{\bf Proof of \autoref{thm:LT}}]
The proof consists of three steps: convergence of the
one-dimensional distributions, the finite-dimensional distributions,
and a tightness argument.

\vb

{\it Step 1: Convergence of one-dimensional distributions}. In this
step we show that, for a fixed $t\ge 0$,
\[
\frac{Q_X^{(c)}(t)}{\sigma(\delta(c))}\dto Q_{B_\lambda}^{(1)}(t),$ as $c\toi.
\]
Since $Q_X^{(c)}$ is stationary, it is enough to show the above
convergence for $t=0$ only. Observe that, due to the
time-reversibility property of Gaussian processes,
\[
Q_X^{(c)}(0)\de\sup_{t\ge 0}\left(X(t)-ct\right) =\sup_{t\ge
0}\left(X(\delta(c)t)-c\delta(c) t\right).
\]
Upon combining \autoref{cor:convLT} with the continuous mapping
theorem,  for each $T>0$,
\[
\sup_{t\in[0,T]}\left(\frac{X(\delta(c)t)-c\delta(c)t}{\sigma(\delta(c))}\right)\dto
\sup_{t\in[0,T]}(B_\lambda(t)-t),$ as $c\toi.
\]
Thus it suffices to show that
\begin{equation}
\label{eq:dummy} \lim_{T\toi}\limsup_{c\toi} \prob{\sup_{t\ge
T}\left(
    \frac{X(\delta(c)t)-c\delta(c)t}{\sigma(\delta(c))}
    \right)\ge \eta}=0,
\end{equation}
for any $\eta>0$. Recall the definition of $X^{(c)}$, so that
\[
\prob{\sup_{t\ge T}\left(
    \frac{X(\delta(c)t)-c\delta(c)t}{\sigma(\delta(c))}
    \right)\ge \eta}
\le \prob{\sup_{t\ge T}
    \frac{|X^{(c)}(t)|}{\eta+t}
    \ge 1},
\]
where we used \eqref{eq:defdelta}. \autoref{thm:convLT} implies that
the family $\{X^{(c)}\}$ is tight in $\Omega^\gamma$, for some
$\gamma\le 1$. Now \eqref{eq:dummy} follows from
\autoref{prop:convO}.

\vb

{\it Step 2: Convergence of finite-dimensional distributions.} The
argumentation of this step is analogous to Step 1. First note that
for any $t_i\ge 0$, $\eta_i>0$ and $s_i<t_i$, where $i=1,\ldots,n$,
for any $n\in\nn$, it follows that
\begin{align*}
\pp\Bigg(&\frac{Q^{(c)}_X(\delta(c)t_i)}{\sigma(\delta(c))}>\eta_i,\,i=1,\ldots,n\Bigg)\\
&=
\prob{\sup_{s\le \delta(c)t_i}\left(\frac{X(\delta(c)t_i)-X(s)-c(\delta(c)t_i-s)}{\sigma(\delta(c))}\right)>\eta_i,\,i=1,\ldots,n}\\
&\le
\prob{\sup_{s\in[s_i,t_i]}\left( \frac{X(\delta(c)t_i)-X(\delta(c)s) -c\delta(c)(t_i-s)}{\sigma(\delta(c))}\right)>\eta_i,i=1,\ldots,n}\\
&\quad +\: \sum_{i=1}^n\prob{\sup_{s\le s_i}\left(
\frac{X(\delta(c)t_i)-X(\delta(c)s)
-c\delta(c)(t_i-s)}{\sigma(\delta(c))}\right)>\eta_i}.
\end{align*}
Now the same procedure can be followed as in Step 1.

\vb

{\it Step 3: Tightness in $C[0,T]$.} In this step, for any $T>0$, we
show the tightness of
$\{Q^{(c)}_X(\delta(c)\cdot)/\sigma(\delta(c))\}$ in $C[0,T]$. Given
that we have established Step 2 already, \eqref{eq:tight1} holds so we
are left with proving \eqref{eq:modulus}, with 
$s,t\in[0,T]$; the remainder of the proof is devoted to settling
this claim.

Stationarity of $Q^{(c)}_X$ implies that $
\{Q^{(c)}_X(\delta(c)t)-Q^{(c)}_X(\delta(c)s):t\ge s\}$ is
distributed as \[\{Q^{(c)}_X(\delta(c)(t-s))-Q^{(c)}_X(0):t\ge s\},
\] so that it suffices to prove \eqref{eq:modulus} for $s=0$ only.
Furthermore, cf.\ \eqref{eq:defQ},
\[
\sup_{0<t\le\zeta}\left|Q_X^{(c)}(\delta(c)t)-Q_X^{(c)}(0)\right|
\le 
2\sup_{0<t\le\zeta}\left|X(\delta(c)t)-c\delta(c) t\right|.
\]
From \autoref{cor:convLT} it follows that
\[
\sup_{0<t\le\zeta}\frac{\left|X(\delta(c)t)-c\delta(c) t\right|}{\sigma(\delta(c))}\dto
\sup_{0<t\le\zeta}\left|B_{\lambda}(t)-t\right|,$ as $c\toi.
\]
Now notice that for $\zeta<\eta/4$, by the self-similarity of $B_\lambda$,
\[
\prob{\sup_{0<t\le\zeta}|B_\lambda(t)-t|\ge\frac{\eta}{2}}
\le 
2\prob{\sup_{0<t\le1}B_\lambda(t)\ge\frac{\eta}{4}\zeta^{-\lambda}}.
\]
Now it is straightforward to conclude that the last expression tends to zero as $\zeta\to 0$.
\end{proof}

\newpage
\small
\bibliography{Gaussian.HT}
\end{document}